\documentclass[12pt]{article}

\usepackage{amsmath}
\usepackage{amssymb}
\usepackage{amsthm}

\newtheorem{thm}{Theorem}

\newtheorem{lem}{Lemma}
\newtheorem{prop}{Proposition}
\newtheorem{cor}{Corollary}

\newtheorem{defn}{Definition}

\title{Almost cubes and fourth powers in short intervals}
\author{Tsz Ho Chan}
\date{}

\begin{document}
\maketitle
\begin{abstract}
In this paper, we study how short an interval $[x, x + x^\theta]$ contains an integer of the form $n_1 n_2 n_3$ and $m_1 m_2 m_3 m_4$ with $n_1 \approx n_2 \approx n_3$ and $m_1 \approx m_2 \approx m_3 \approx m_4$. The new idea is to adopt a second moment method (usually used for almost all results) to deduce a result for all short intervals.
\end{abstract}

\section{Introduction and main results}

In a series of work \cite{C2} - \cite{C7}, the author studied almost squares in short intervals. For the purpose of this paper, we say that a positive integer $n$ is an {\it almost square} if it can be factored as $n = n_1 n_2$ where $c \sqrt{n} \le n_1, n_2 \le C \sqrt{n}$ for some fixed constants $0 < c < 1 < C$. The best known result is that, for some constant $c_2 > 1$, the interval $[x, x + c_2 x^{1/4}]$ always contains an almost square for all sufficiently large $x$. It is based on the following elementary observation: If $m^2 < x \le (m + 1)^2$ for some integer $m \ge 0$, then
\[
x - m^2, (m+1)^2 - x < (m+1)^2 - m^2 = 2m + 1 \le 2 \sqrt{x} + 1.
\]
More specifically, for $x > 1$, one can choose
\[
a = \lceil \sqrt{x} \rceil \; \; \text{ and } \; \; b = \lfloor \sqrt{ a^2 - x } \rfloor.
\]
Then $0 \le a^2 - x \ll \sqrt{x}$ and $0 \le (a^2 - x) - b^2 \ll \sqrt{a^2 - x} \ll \sqrt[4]{x}$ which yields $n = n_1 n_2 := (a - b) (a + b) \in [x, x + c_2 x^{1/4}]$ as
\begin{equation} \label{2factor}
n_1 n_2 - x = (a - b) (a + b) - x = (a^2 - x) - b^2 \ll \sqrt[4]{x}.
\end{equation}

One can generalize the above concept of almost squares to almost cubes or almost $k$-th powers as follows. Given an integer $k \ge 2$, we say that a positive integer $n$ is an {\it almost $k$-th power} if it can be factored as
\[
n = n_1 n_2 \cdots n_k \; \; \text{ with } \; \; n_1, n_2,  \ldots, n_k \asymp
n^{1/k}.
\]
In this article, we are interested in studying how short an interval $[x, x + x^{\theta_k}]$ with $0 < \theta_k < 1$ contains an almost $k$-th power. Here and throughout the paper, we assume that $x > 1$ is sufficiently large. 

\bigskip

For $k = 3$, one can apply the above elementary method to show that any $\theta_3 > 1/2$ works. For example, choose $n_3 = \lceil x^{1/3} \rceil$. Then, $\frac{x}{n_3} \asymp x^{2/3}$ and one can find $n_1, n_2 \asymp x^{1/3}$ such that $0 \le n_1 n_2 - \frac{x}{n_3} \ll (\frac{x}{n_3})^{1/4} \ll x^{1/6}$ by \eqref{2factor}. Hence, $0 \le n_1 n_2 n_3 - x \ll x^{1/6} \cdot n_3 \ll x^{1/2}$.

\bigskip

Recently, in connection with elliptic curve cryptography, Islam \cite{Is} showed the existence of a number $n = n_1 n_2 n_3 \in [x, x + x^{1/2}]$ where $n_1, n_2, n_3 \asymp n^{1/3}$ are pairwise relatively prime. Moreover, Islam's argument allows one of the $n_1, n_2, n_3$ to be a prime number. So, one may ask if it is possible to restrict two or even all three of the factors to be prime numbers. This is then related to the study of almost primes $P_k$ (numbers with up to $k$ prime factors) and, more precisely, $E_k$ (numbers with exactly $k$ prime factors) in short intervals. Recently, Matom\"{a}ki and Ter\"{a}v\"{a}inen \cite{MT} proved that $[x, x + \sqrt{x} (\log x)^{1.55}]$ contains some $E_3$ numbers. However, in their result, one of the prime factors is very small with size $\asymp (\log x)^{1.1}$. So, it remains a challenge to find short intervals $[x, x + x^\theta]$ that contain $p_1 p_2 p_3$ with primes $p_1, p_2, p_3 \asymp x^{1/3}$.

\bigskip

Beside prime numbers, one can restrict the factors $n_1, n_2, n_3$ to other interesting arithmetic sequences or even random sequences. Another direction is to look at almost $k$-th powers with $k \ge 4$. Towards these, we can have some general and specific results. First, let us make the following definition.
\begin{defn}
We say that an infinite sequence of positive integers $\mathcal{A}$ is \emph{``almost dense"} if, for any $\epsilon > 0$, there exists a constant $c_{\epsilon, A} > 0$ such that
\begin{equation} \label{denseA}
\# \{ X \le n \le 2 X  \, : \, n \in \mathcal{A} \} \ge c_{\epsilon, \mathcal{A}} X^{1 - \epsilon / 4}
\end{equation}
for all sufficiently large $X$.
\end{defn}

\begin{thm} \label{thm1}
For any $\epsilon > 0$ and any two ``almost dense" sequences $\mathcal{A}_1$ and $\mathcal{A}_2$, the interval $[x, x + x^{5/9 + \epsilon}]$ contains an integer $n = m \cdot a_1 \cdot a_2$ for some $a_1 \in \mathcal{A}_1$, $a_2 \in \mathcal{A}_2$ and integer $m$ with $a_1, a_2, m \asymp x^{1/3}$ for all sufficiently large $x$.
\end{thm}
\begin{cor}
For any $\epsilon > 0$ and all sufficiently large $x$, the interval $[x, x + x^{5/9 + \epsilon}]$ contains an integer $n = p_1 \cdot p_2 \cdot m$ for some primes $p_1, p_2$ and integer $m$ with $p_1, p_2, m \asymp x^{1/3}$.
\end{cor}
\begin{thm} \label{thm2}
For any $\epsilon > 0$ and any three ``almost dense" sequences $\mathcal{A}_1$, $\mathcal{A}_2$ and $\mathcal{A}_3$, the interval $[x, x + x^{34/55 + \epsilon}]$ contains an integer $n = m \cdot a_1 \cdot a_2 \cdot a_3$ for some $a_i \in \mathcal{A}_i$ (for $i = 1, 2, 3$) and integer $m$ with $a_1, a_2, a_3, m \asymp x^{1/4}$ for all sufficiently large $x$.
\end{thm}
\begin{cor}
For any $\epsilon > 0$ and all sufficiently large $x$, the interval $[x, x + x^{34/55 + \epsilon}]$ contains an integer $n = p_1 \cdot p_2 \cdot p_3 \cdot m$ for some primes $p_1, p_2, p_3$ and integer $m$ with $p_1, p_2, p_3, m \asymp x^{1/4}$. Note that $34 / 55 = 0.61818\ldots$.
\end{cor}

Similar to \cite{C}, we can obtain the following almost all result for almost cubes.
\begin{thm} \label{thm3}
Let $\epsilon > 0$ and $X > 0$ be sufficiently large. Then, for almost all $x \in [X, 2X]$, the interval $[x, x + x^{13/55 + \epsilon}]$ contains an almost cube $n = n_1 n_2 n_3$ with $x^{1/3} / 2 \le n_1, n_2, n_3 \le 2 x^{1/3}$. Here almost all means apart from a set of measure $o(X)$. Note that $13/55 = 0.23636\ldots$.
\end{thm}
Assuming the Lindel\"{o}f hypothesis, we have the following conditional close-to-optimal result.
\begin{thm} \label{thm4}
Let $\epsilon > 0$ and $X > 0$ be sufficiently large. Under the Lindel\"{o}f hypothesis, for almost all $x \in [X, 2X]$, the interval $[x, x + x^{\epsilon}]$ contains an almost cube $n = n_1 n_2 n_3$ with $x^{1/3} / 2 \le n_1, n_2, n_3 \le 2 x^{1/3}$.
\end{thm}

{\bf Notation.} The symbols $f(x) = O(g(x))$, $f(x) \ll g(x)$, and $g(x) \gg f(x)$ are equivalent to $|f(x)| \leq C g(x)$ for some constant $C > 0$. The symbol $f(x) \asymp g(x)$ means that $f(x) \ll g(x) \ll f(x)$. Also, $f(x) = O_{\lambda} (g(x))$ and $f(x) \ll_{\lambda} g(x)$ mean that the implicit constant may depend on the parameter $\lambda$. The symbol $f(x) = o(g(x))$ means $\lim_{x \rightarrow \infty} \frac{f(x)}{g(x)} = 0$. The ceiling function $\lceil x \rceil$ is the least integer greater than or equal to $x$ while the floor function $\lfloor x \rfloor$ is the greatest integer less than or equal to $x$.

\section{Some preparations}

We employ the method in \cite{C} which was inspired by Soundararajan \cite{S}. Let $\mathcal{A}$ be an ``almost dense" infinite sequence of positive integers.  Suppose $a_n$ is a set of coefficients such that, for all $\epsilon > 0$, $1 \le a_n \ll_\epsilon n^\epsilon$ if $n \in \mathcal{A}$ and $a_n = 0$ if $n \not\in \mathcal{A}$. With $1 \le L \le U/2 < U$, define
\[
A(s) := \mathop{\sum_{U - L \le n \le U + L}}_{n \in \mathcal{A}} \frac{a_n}{n^s}.
\]
It follows from \eqref{denseA} that
\begin{equation} \label{denseA2}
A(1) \gg_\epsilon U^{-\epsilon / 4}
\end{equation}
by partial summation. Let $\mathcal{B}$ be another infinite sequence of positive integers with weights $b_m \ll_\epsilon m^\epsilon$ for all $\epsilon > 0$ and $m \in \mathcal{B}$ such that its Dirichlet series
\[
B(s) := \sum_{m \in \mathcal{B}} \frac{b_m}{m^s} \; \text{ has a pole at } s = 1 \text{ with  residue } r_{\mathcal{B}} > 0
\]
and an analytic continuation to the half-plane $\Re s > 0$ with no other poles. Consider
\[
\Phi(w) := \frac{1}{2 \pi i} \int_{\sigma - i \infty}^{\sigma + i \infty} B(s) A(s) w^s \frac{(e^{\delta s} - 1)^2}{s^2} ds.
\]
with $\sigma = 1 + 1 / \log x$. By shifting contour and Cauchy residue theorem, one has
\[
\frac{1}{2 \pi i} \int_{\sigma - i \infty}^{\sigma + i \infty} \xi^s \frac{ds}{s^2} = \left\{ \begin{array}{ll} \log \xi, & \text{ if } \xi \ge 1, \\
0, & \text{ if } 0 < \xi \le 1 \end{array} \right.
\]
and
\[
\frac{1}{2 \pi i} \int_{\sigma - i \infty}^{\sigma + i \infty} \xi^s \frac{(e^{\delta s} - 1)^2}{s^2} ds = \left\{ \begin{array}{ll} \min( \log (e^{2\delta} \xi), \log(1/\xi)), & \text{ if } e^{-2 \delta} \le \xi \le 1, \\
0, & \text{ otherwise}. \end{array} \right.
\]
Thus,
\[
\Phi(w) = \mathop{\mathop{\sum_{n \in \mathcal{A}, m \in \mathcal{B}}}_{w \le m n \le w e^{2 \delta}}}_{U - L \le n \le U + L} b_m a_n \min \Bigl( \log \frac{e^{2 \delta} w}{m n}, \log \frac{m n}{w} \Bigr).
\]

\bigskip

Now we shift the line of integration in the definition of $\Phi(w)$ to the left. By Cauchy residue theorem,
\begin{equation} \label{integral}
\Phi(w) - r_{\mathcal{B}} \cdot w (e^\delta - 1)^2 A(1) = \frac{-1}{2 \pi i} \int_{\eta - i \infty}^{\eta + i \infty} B(s) A(s) w^s \frac{(e^{\delta s} - 1)^2}{s^2} ds.
\end{equation}
for $\eta := 1/2$. The new idea is to study the following second moment
\begin{equation} \label{I}
I := \int_{100 \Delta}^{200 \Delta} \Big| \Phi \Bigl( \frac{x}{c} \Bigr) - \frac{x}{c} (e^\delta - 1)^2 A(1) r_{\mathcal{B}} \Big|^2 dc
\end{equation}
and show that it is small.
\begin{prop} \label{prop1}
For any $\epsilon > 0$ and $0 < \delta < 1$,
\[
I \ll_\epsilon x \delta^{2} \sum_{j \ge 0} \Bigl( \frac{\delta}{2^j} \Bigr)^{2 - \epsilon / 4} \int_{2^{j-1} / \delta}^{2^j / \delta} \Big| B (\eta + i t) A(\eta + i t) \Big|^2 dt.
\]
\end{prop}
Suppose $c \in [100 \Delta, 200 \Delta]$ is an integer such that $\Phi(\frac{x}{c}) = 0$. As $\frac{x}{c} - \frac{x}{c + \theta} = \frac{\theta x}{c (c + \theta)} \le \frac{\theta x}{\Delta^2}$ for $0 \le \theta \le 1$,
\[
\Big| \Phi \Big(\frac{x}{c}\Bigr) - \Phi \Bigl(\frac{x}{c + \theta} \Bigr) \Big| \ll_\epsilon \frac{4 \theta x}{\Delta^2} \cdot \Delta^{\epsilon / 4} U^{\epsilon / 4} \delta
\]
which is less than
\begin{equation} \label{shortcond}
\frac{x}{2 c} (e^\delta - 1)^2 A(1) r_{\mathcal{B}} \; \; \text{ when } \; \; \theta \le c_{\epsilon} A(1) r_{\mathcal{B}} \cdot \frac{\Delta^{1 - \epsilon / 4} \delta}{U^{\epsilon / 4}}
\end{equation}
for some small constant $c_\epsilon > 0$. Let $C$ denote the number of integers $c \in [100 \Delta, 200 \Delta]$ such that $\Psi(\frac{x}{c}) = 0$. Combining \eqref{denseA2}, \eqref{I} and \eqref{shortcond} with Proposition \ref{prop1}, we obtain
\begin{equation} \label{sizeC}
C \ll_{\epsilon, \mathcal{A}, \mathcal{B}} \frac{U^\epsilon \Delta^{1 + \epsilon / 4}}{x \delta^3}  \sum_{j \ge 0} \Bigl( \frac{\delta}{2^j} \Bigr)^{2 - \epsilon / 4} \int_{2^{j-1} / \delta}^{2^j / \delta} \Big| B (\eta + i t) A(\eta + i t) \Big|^2 dt.
\end{equation}
Our eventual goal is to show that $C = o(\Delta)$.

\bigskip

Now, we recall a mean-value theorem and a majorant principle for Dirichlet polynomials.

\begin{lem} \label{lem-mv}
Let $D(s) = \sum_{n = 1}^{N} d_n n^{-s}$ be a Dirichlet polynomial. Then
\[
\int_{0}^{T} |D(i t)|^2 dt = (T + O(N)) \sum_{n = 1}^{N} |d_n|^2.
\]
\end{lem}
\begin{proof}
See \cite[Chapter 7, Theorem 1]{M} for example.
\end{proof}

\begin{lem} \label{lem-maj}
Suppose $|d_n| \le D_n$ for $1 \le n \le N$. Then
\[
\int_{-T}^{T} \Big| \sum_{n = 1}^{N} \frac{d_n}{n^{i t}} \Big|^2 \, dt \le 3 \int_{-T}^{T} \Big| \sum_{n = 1}^{N} \frac{D_n}{n^{i t}} \Big|^2 \, dt.
\]
\end{lem}
\begin{proof}
See \cite[Chapter 7, Theorem 3]{M} for example.
\end{proof}

The following estimates related to the Riemann zeta function $\zeta(s)$ are also needed.
\begin{lem} \label{lem-bour}
For any $\epsilon > 0$, $\zeta(\eta + i t) \ll_\epsilon (|t| + 1)^{13 / 84 + \epsilon}$.
\end{lem}

\begin{proof}
This is a recent breakthrough result of Bourgain \cite{B} via decoupling method.
\end{proof}

\begin{lem} \label{lem1}
For $T \ge 1$ and $U^{1/2} < L \le U / 2$,
\[
\int_{1}^{T} \big| \zeta ( \eta + i t) A ( \eta + i t ) \big|^2 \, dt \ll \frac{T L}{U} \log^2 TU + \frac{T^{1/2} L^2}{U} \log TU.
\]
\end{lem}

\begin{proof}
This is essentially Lemma 3.3 in \cite{C} after incorporating the majorant principle, Lemma \ref{lem-maj}, into it.
\end{proof}

Finally, we recall the following simple integral bound.
\begin{lem} \label{lem-logbd}
For $u \ge 0$ and $X \ge 1$,
\[
\int_{0}^{X} \frac{dv}{1 + |v - u|} \le \log (1 + u) + \log(1 + X).
\]
\end{lem}

\begin{proof}
This is Lemma 3.4 in \cite{C}.
\end{proof}

\section{Proof of Proposition \ref{prop1}}

Putting \eqref{integral} into \eqref{I}, we have
\begin{equation} \label{I2}
I = \frac{1}{4 \pi^2} \int_{100 \Delta}^{200 \Delta} \Big| \int_{- \infty}^{\infty} B (\eta + i t) A(\eta + i t)  \Bigl( \frac{x}{c} \Bigr)^{\eta + i t} \frac{(e^{\delta (\eta + it)} - 1)^2}{(\eta + it)^2} dt \Big|^2 dc.
\end{equation}
Expanding \eqref{I2} out and integrating over $c$, we have
\begin{align*}
I \ll& x \int_{-\infty}^{\infty} \int_{-\infty}^{\infty} \Big| B (\eta + i t) B (\eta + i s) A(\eta + i t) A(\eta + i s) \Big| \\
&\times \min \Bigl( \delta^2, \frac{1}{t^2} \Bigr) \min \Bigl( \delta^2, \frac{1}{s^2} \Bigr) \frac{1}{1 + |t - s|} \, dt \; ds
\end{align*}
as
\[
\frac{(e^{\delta (\eta + it)} - 1)^2}{(\eta + it)^2} \ll \min \Bigl( \delta^2, \frac{1}{|t|^2} \Bigr) \; \; \text{ and } \; \; \int_{100 \Delta}^{200 \Delta} c^{-2 \eta + i (t - s)} dc \ll \frac{1}{1 + |t - s|}.
\]
By the inequality $2 a b \le a^2 + b^2$,
\[
|B (\eta + i t) B (\eta + i s) A(\eta + i t) A(\eta + i s)| \ll |(B \cdot A) (\eta + it)|^2 + |(B \cdot A) (\eta + i s)|^2.
\]
Hence, by symmetry,
\[
I \ll x \int_{-\infty}^{\infty} \int_{-\infty}^{\infty} \Big| B (\eta + i t) A(\eta + i t) \Big|^2  \min \Bigl( \delta^2, \frac{1}{t^2} \Bigr) \min \Bigl( \delta^2, \frac{1}{s^2} \Bigr) \frac{1}{1 + |t - s|} \, dt \; ds
\]
Next, we split the interval $(-\infty, \infty)$ into subintervals $\mathcal{I}_0 = \{ |t| \le 1/\delta \}$, and $\mathcal{I}_{j} = \{ 2^{j-1}/\delta \le |t| \le 2^j/\delta \}$ for $j \in \mathbb{N}$. By Lemma \ref{lem-logbd} and $\log x \ll_\epsilon x^{\epsilon/4}$ for $x \ge 1$,
\begin{align} \label{Ineq}
I \ll& \, x \sum_{j, j' \ge 0} \, \int_{\mathcal{I}_{j}} \int_{\mathcal{I}_{j'}} \Big| B (\eta + i t) A(\eta + i t) \Big|^2 \min \Bigl( \delta^2, \frac{1}{t^2} \Bigr) \min \Bigl( \delta^2, \frac{1}{s^2} \Bigr) \frac{1}{1 + |t - s|} dt \; ds \nonumber \\
\ll& \, x \sum_{j, j' \ge 0} \Bigl(\frac{\delta}{2^{j}}\Bigr)^2 \Bigl(\frac{\delta}{2^{j'}}\Bigr)^2 \log \frac{2^{\max(j,j')}}{\delta} \int_{\mathcal{I}_j} \Big| B (\eta + i t) A(\eta + i t) \Big|^2 dt \nonumber \\
\ll_\epsilon& \, x \delta^{2} \sum_{j \ge 0} \Bigl( \frac{\delta}{2^j} \Bigr)^{2 - \epsilon / 4} \int_{2^{j-1} / \delta}^{2^j / \delta} \Big| B (\eta + i t) A(\eta + i t) \Big|^2 dt
\end{align}
which gives the proposition.

\section{Proof of Theorem \ref{thm1}}

In this section, we specialize $B(s) = \zeta(s)$, the Riemman zeta function, and set $\Delta = x^{1/3}$, $U = 0.1 x^{1/3}$ and $L = 0.1 U$. Then $r_{\mathcal{B}} = 1$. Putting Lemma \ref{lem1} into \eqref{sizeC}, we have
\[
C \ll_{\epsilon, \mathcal{A}, \mathcal{B}} \frac{\Delta}{x^{1 - \epsilon} \delta^2} + \frac{\Delta U}{x^{1 - \epsilon} \delta^{3/2}} \ll \Delta^{1 - \epsilon / 3}
\]
when $\delta = x^{- 4/9 + \epsilon}$. Hence, in view of of \eqref{denseA}, we have some $c \in \mathcal{A}_2$ such that $\Phi(x / c) \neq 0$. This means that there is some integer of the form $m \cdot a$ with $a \in \mathcal{A}_1$ such that $\frac{x}{c} \le m a \le \frac{x}{c} e^{2 \delta}$. Therefore, $x \le m a c \le x e^{2 \delta} = x + O(x^{5/9 + \epsilon})$. Since $\epsilon > 0$ can be arbitrarily small, we have Theorem \ref{thm1}.

\section{Proof of Theorem \ref{thm2}}

In this section, we specialize $B(s) = \zeta(s)$ and set $\Delta = x^{1/4}$, $U = 0.1 x^{1/4}$ and $L = 0.1 U$. Let $\mathcal{A}_1$ and $\mathcal{A}_2$ be two ``almost dense" infinite sequences of positive integers. Define
\[
\mathcal{A} := \{ a \cdot b \, : \, a, b \in \mathcal{A}_1 \times \mathcal{A}_2 \}.
\]
One can check that the condition \eqref{denseA} is satisfied by looking at the subset $\{ a \cdot b \, : \, a, b \in \mathcal{A}_1 \times \mathcal{A}_2 \text{ and } \frac{b}{\sqrt{2}} \le a \le \sqrt{2} b \}$, applying the condition \eqref{denseA} for $\mathcal{A}_1$ and $\mathcal{A}_2$ with $\epsilon/3$ in place of $\epsilon$, and the fact that $d(n) \ll_\epsilon n^{\epsilon /12}$. Let
\[
a_n = \# \Bigl\{ (a, b) \in \mathcal{A}_1 \times \mathcal{A}_2 \, : \, a \cdot b = n \Bigr\}
\]
By Lemmas \ref{lem-bour} and \ref{lem-maj}, we have
\begin{align*}
S :=& \sum_{j \ge 0} \Bigl( \frac{\delta}{2^j} \Bigr)^{2 - \epsilon / 4} \int_{2^{j-1} / \delta}^{2^j / \delta} \Big| \zeta (\eta + i t) A(\eta + i t)^2 \Big|^2 dt \\
\ll_\epsilon& \sum_{j \ge 0} \Bigl( \frac{\delta}{2^j} \Bigr)^{2 - 13/42 - \epsilon / 3} \int_{2^{j-1} / \delta}^{2^j / \delta} \Big| A_0(\eta + i t) \Big|^4 dt
\end{align*}
where $A_0(s) = \sum_{1 \le n \le 2 \sqrt{U}} n^{-s}$. By Lemma \ref{lem-mv} with $D(s) = A_0^2(s)$,
\[
S \ll_\epsilon x^{\epsilon / 8} \sum_{j \ge 0} \Bigl( \frac{\delta}{2^j} \Bigr)^{2 - 13/42 - \epsilon / 3} \Bigl( \frac{2^j}{\delta} + U^2 \Bigr) \ll_\epsilon x^{\epsilon / 8} \delta^{29/42 - \epsilon / 3} +  x^{1/2 + \epsilon / 8} \delta^{71/42 - \epsilon / 3}.
\]
Putting this into \eqref{sizeC}, we get
\[
C \ll_{\epsilon, \mathcal{A}, \mathcal{B}} \frac{\Delta}{x^{1 - \epsilon} \delta^{97 / 42}} + \frac{\Delta}{x^{1/2 - \epsilon} \delta^{55 / 42}} \ll \Delta^{1 - \epsilon / 4}
\]
when $\delta = x^{- 21/55 + \epsilon}$. Hence, in view of of \eqref{denseA}, we have some $c \in \mathcal{A}_3$ such that $\Phi(\frac{x}{c}) \neq 0$. This means that there is some integer of the form $m \cdot a_1 \cdot a_2$ with $a_1 \in \mathcal{A}_1$, $a_2 \in \mathcal{A}_2$ such that $\frac{x}{c} \le m a_1 a_2 \le \frac{x}{c} e^{2 \delta}$. Therefore, $x \le m a_1 a_2 c \le x e^{2 \delta} = x + O(x^{34 / 55 + \epsilon})$. Since $\epsilon > 0$ can be arbitrarily small, we have Theorem \ref{thm2}.

\section{Proof of Theorems \ref{thm3} and \ref{thm4}}

With $U = x^{1/3}$ and $L = U/2$, we apply \eqref{integral} with $B(s) = \zeta(s)$ and $A(s) = N(s)^2$ where
\[
N(s) = \sum_{U - L \le n \le U + L} \frac{1}{n^s}.
\]
Similar to \cite{C}, we consider the following second moment
\begin{align*}
J :=& \int_{x}^{2x} \Big| \Phi(y) - y (e^\delta - 1)^2 N(1)^2 \Big|^2 dy \\
=& \frac{1}{4 \pi^2} \int_{x}^{2x} \Big| \int_{-\infty}^{\infty} \zeta( \eta + i t) N( \eta + it )^2 x^{\eta + it} \frac{(e^{\delta (\eta + i t)} - 1)^2}{(\eta + it)^2} dt \Big|^2 dy.
\end{align*}
Similar to the proof of Proposition \ref{prop1}, we expand things out, integrate over $y$ and apply symmetry to get
\begin{align*}
J &\ll \, x^2 \int_{-\infty}^{\infty} \int_{-\infty}^{\infty} | \zeta( \eta + i t ) |^2 | N( \eta + i t) |^4 \min \Bigl( \delta^2, \frac{1}{t^2} \Bigr) \Bigl( \delta^2, \frac{1}{s^2} \Bigr) \frac{1}{1 + |t - s|} dt \, ds \\
&\ll_\epsilon \, x^2 \delta^2 \sum_{j \ge 0} \Bigl(\frac{\delta}{2^j} \Bigr)^{2 - \epsilon / 4} \int_{2^{j-1} / \delta}^{2^j / \delta} | \zeta( \eta + i t ) |^2 | N( \eta + i t) |^4 dt.
\end{align*}
Applying Lemma \ref{lem-bour} and Lemma \ref{lem-mv} with $D(s) = N(s)^2$ to the above, we have
\begin{align} \label{J}
J &\ll \, x^2 \delta^2 \sum_{j \ge 0} \Bigl(\frac{\delta}{2^j} \Bigr)^{2 - 13/42 - \epsilon / 3} \int_{2^{j-1} / \delta}^{2^j / \delta} | N( \eta + i t) |^4 dt \\
&\ll \, x^{2 + \epsilon / 8} \delta^2 \sum_{j \ge 0} \Bigl( \frac{\delta}{2^j} \Bigr)^{2 - 13/42 - \epsilon / 3} \Bigl( \frac{2^j}{\delta} + U^2 \Bigr) \nonumber \\
&\ll_\epsilon \, x^{2 + \epsilon / 8} \delta^{113/42 - \epsilon / 3} +  x^{8/3 + \epsilon / 8} \delta^{155/42 - \epsilon / 3}. \nonumber
\end{align}
Now, let
\[
\mathcal{S} := \{ y \in [x, 2x] : \Phi(y) = 0 \}.
\]
Then
\[
| \mathcal{S} | \cdot x^2 \delta^4 \ll_\epsilon x^{2 + \epsilon / 8} \delta^{113/42 - \epsilon / 3} +  x^{8/3 + \epsilon / 8} \delta^{155/42 - \epsilon / 3}
\]
and
\[
| \mathcal{S} | \ll_\epsilon \frac{x^{\epsilon / 8}}{\delta^{55 / 42 + \epsilon / 3}} + \frac{x^{2/3 + \epsilon / 8}}{\delta^{13 / 42 + \epsilon / 3}} = o(x)
\]
when $\delta = x^{-42/55 + \epsilon / 2}$. Hence, $\Phi(y) \neq 0$ for almost all $y \in [x, 2x]$. This gives Theorem \ref{thm3} as $y \cdot e^{2 \delta} \le y (1 + 3 \delta) \le y + y^{13/55 + \epsilon}$ for $x$ sufficiently large.

\bigskip

If one assumes the Lindel\"{o}f hypothesis, we have $\zeta(\eta + i t) \ll_\epsilon t^{\epsilon / 3}$ instead of Bourgain's bound in Lemma \ref{lem-bour}. Hence, one can remove the $-13/42$ part of the exponent in \eqref{J} and obtain
\[
J \ll_\epsilon \frac{x^{2 + \epsilon / 8}}{\delta^{3 + \epsilon / 3}} + \frac{x^{8 / 3 + \epsilon / 8}}{\delta^{4 + \epsilon / 3}}
\]
instead. By a similar calculation as the unconditional situation, we have
\[
| \mathcal{S} | \ll_\epsilon \frac{x^{\epsilon / 8}}{\delta^{1 + \epsilon / 3}} + \frac{x^{2 / 3 +  \epsilon / 8}}{\delta^{\epsilon / 3}} = o(x)
\]
when $\delta = x^{-1 + \epsilon / 2}$. This yields Theorem \ref{thm4}.

\bigskip

{\bf Acknowledgment.} The author would like to thank Chi Hoi Yip for some stimulating discussion.
 
Department of Mathematics \\
Kennesaw State University \\
Marietta, GA 30060 \\
tchan4@kennesaw.edu

\end{document}